\documentclass{amsart}

\usepackage{amsmath,amssymb,amsthm}
\usepackage{enumitem}
\usepackage{comment}
\usepackage{stmaryrd}
\usepackage{textcomp}

\usepackage{fullpage}
\usepackage{tikz}
\usetikzlibrary{matrix}
\usepackage{tikz-cd}

\newcommand{\W}{\mathbb{W}}
\newcommand{\F}{\mathbb{F}}

\renewcommand{\a}{\mathfrak{a}}

\newcommand{\G}{\mathbb{G}}
\newcommand{\M}{\mathcal{M}}
\newcommand{\Z}{\mathbb{Z}}

\newcommand{\CP}{\mathbb{C}\mathrm{P}}
\newcommand{\A}{\hat{\mathbb{A}}}
\newcommand{\N}{\mathbb{N}}

\newcommand{\<}{\langle}
\renewcommand{\>}{\rangle}

\renewcommand{\)}{)\!)}
\renewcommand{\phi}{\varphi}
\renewcommand{\epsilon}{\varepsilon}

\newcommand{\powser}[1]{\llbracket {#1} \rrbracket}

\newcommand{\CatOf}[1]{\mathsf{#1}}
\newcommand{\real}[1]{\underline{\smash{#1}}}
\newcommand{\portrait}{\square}

\DeclareMathOperator{\Spec}{Spec}
\DeclareMathOperator{\Spf}{Spf}

\DeclareMathOperator*{\colim}{colim}
\DeclareMathOperator{\Hom}{Hom}

\DeclareMathOperator{\Pro}{Pro}
\DeclareMathOperator{\Ind}{Ind}
\DeclareMathOperator{\Rings}{Rings}
\DeclareMathOperator{\id}{id}
\DeclareMathOperator{\Spp}{Spp}

\theoremstyle{plain}
\newtheorem{theorem}{Theorem}
\newtheorem*{maintheorem}{Main Theorem}
\newtheorem{proposition}[theorem]{Proposition}
\newtheorem{lemma}[theorem]{Lemma}
\newtheorem{corollary}[theorem]{Corollary}
\newtheorem{conjecture}[theorem]{Conjecture}
\theoremstyle{definition}
\newtheorem{definition}[theorem]{Definition}
\theoremstyle{remark}
\newtheorem{remark}[theorem]{Remark}
\newtheorem{example}[theorem]{Example}

\title{A relative Lubin--Tate theorem via meromorphic formal geometry}
\author{Aaron Mazel-Gee \\ Eric Peterson \\ Nathaniel Stapleton}

\begin{document}

\maketitle

\begin{abstract}
We formulate a theory of punctured affine formal schemes, suitable for certain problems within algebraic topology.  As an application, we show that the Morava $K$-theoretic localizations of Morava $E$-theory corepresent a Lubin--Tate-type moduli problem in this framework.
\end{abstract}

\section{Introduction}

Associated to a fixed formal group law $\Gamma$ of finite height $n$ over the field $\F_{p^n}$, there is a cohomology theory $E_n$ called Morava $E$-theory.  Its coefficient ring can be noncanonically identified as \[E_n^0 \cong \W_{\F_{p^n}}\llbracket u_1, \ldots, u_{n-1}\rrbracket.\]  Both $E$-theory and this ring are intimately connected to the infinitesimal deformation theory of formal groups: the associated formal scheme $LT_n = \Spf E_n^0$ is known as Lubin--Tate space, which classifies deformations of $\Gamma$~\cite{LubinTate}.  More precisely, when $R$ is a complete local ring with residue field $\F_{p^n}$, maps $\Spf R \to LT_n$ are in bijective correspondence with $\star$-isomorphism classes of formal group laws over $R$ which reduce over the special fiber to the specified formal group law $\Gamma$.  Given a complex orientation $\Spf E_n^0 \mathbb{C}\mathrm{P}^\infty \cong \Spf E_n^0\powser{x}$ of $E$-theory, the resulting formal group law associated to the formal group \[\Spf E_n^0\mathbb{C}\mathrm{P}^\infty \to \Spf E_n^0\] is an example of such a versal formal group law.

Transchromatic geometry studies phenomena related to the self-interaction of the height filtration of the moduli of formal groups as appearing in chromatic homotopy theory.  One object of interest in this subfield is $L_{K(t)} E_n$, the $K(t)$-localization of Morava $E$-theory, where Morava $K$-theory $K(t)$ is the reduction of $E_t$ to the special fiber of $LT_t$.  For $t > n$ the spectrum $L_{K(t)} E_n$ is zero, and for $0 \le t \le n$ it follows from work of Mark Hovey~\cite[Lemma 2.3]{Hovey} that the coefficient ring of $L_{K(t)} E_n$ can be expressed as \[(L_{K(t)} E_n)^0 \cong \W_{\F_{p^n}}\llbracket u_1, \ldots, u_{n-1}\rrbracket[u_t^{-1}]^\wedge_{I_t},\] where $I_t$ denotes the ideal $I_t = (p, u_1, \ldots, u_{t-1})$.  Haynes Miller has suggested that this ring should be similarly studied through the moduli problem it represents --- but from the point of view of formal geometry, this ring is extremely unfriendly.  Here is an inexhaustive list of complaints:
\begin{itemize}
\item It is not a complete local ring, so cannot give rise to a formal scheme in the most naive, 1970s sense of the phrase that has served an enormous amount of algebraic topology perfectly well.
\item Inverting topologically nilpotent elements destroys adic topologies.  The adic topology on $k\llbracket x \rrbracket$ comes from declaring linear translates of powers of the ideal $(x)$ to be a basis of open sets.  In turn, this has the effect of requiring that a continuous map $k\llbracket x \rrbracket \to R$ to a discrete ring $R$ send $x$ to a nilpotent element, as continuous maps respect the ideals of definition.  There are therefore no maps $k\(x\) \to R$ restricting to continuous maps $k\llbracket x \rrbracket \to k\(x\) \to R$, which is highly unsatisfactory when identifying schemes with their functors of points.
\item It is generally only complete against \emph{some} of the power series generators.  So, even making using of the available completeness, when $t < n-1$ there remain power series generators for which the ring is \emph{not} complete.  Power series rings without their completion taken into account are quite intimidating.
\end{itemize}
For all these reasons, the classical framework of formal geometry fails to account for the subtleties of this ring.

Nevertheless, rings of this sort have been gaining importance in chromatic homotopy theory.  Here's a small sampling of applications:
\begin{itemize}
\item Takeshi Torii has worked extensively with the cohomology theories $L_{K(n-1)} E_n$ (i.e., the case $t = n-1$), producing many interesting results about transchromatic phenomena and interrelationships among the stabilizer groups of varying heights~\cite{Torii1,Torii2,Torii4,Torii3,Torii5}.
\item Tyler Lawson and Niko Naumann have encountered these rings while studying the interplay of the $K(1)$- and $K(2)$-local obstructions to realizing $BP\<2\>$ as an $E_\infty$-ring spectrum~\cite{LawsonNaumann}.
\item Such rings appear in and around the Tate construction, especially in connection to results about the pro-spectrum $\CP^\infty_{-\infty}$ arising from stunted projective space.  For example, this powers work of Matthew Ando, Jack Morava, and Hal Sadofsky on a phenomenon known as ``blueshift''~\cite{AndoMorava, AMS}.
\item Punctured formal geometry and Laurent series rings appear in the ``sharp construction'' of Matthew Ando, Christopher French, and Nora Ganter, which connects $\CP^\infty_{-\infty}$ to the study of $MU\<2k\>$-orientations~\cite{AFG}.
\item John Greenlees and Neil Strickland have studied a form of transchromatic character theory involving the ring $E_n[u_t^{-1}] / I_t$, which corresponds to the functions on the punctured neighborhood $X_t \setminus X_{t-1}$ for a filtration $X_*$ on Lubin--Tate space~\cite{GreenleesStrickland}.
\item The third author has developed a variety of results concerning transchromatic character theory and group-cohomological data for the theories $L_{K(t)} E_n$~\cite{StapletonCharacters,StapletonTwist}.  The $p$-divisible group $(\Spf E_n^0 \CP^\infty)[p^\infty] \otimes L_{K(t)} E_n^0$ plays a crucial role.
\end{itemize}
It is therefore important to sort out what variation on formal geometry houses these objects, especially given the success of formal geometry at organizing and interpreting results in algebraic topology.

Our goal is to propose a framework in which the ring $L_{K(t)} E_n^0$ and various natural maps concerning it can be studied.  Kazuya Kato has already ventured in this direction~\cite{Kato}, as inspired by work of Shuji Saito~\cite{Saito1, Saito2}.  Both of them worked in a number theoretic context, and what we contribute lies atop their substantial groundwork.  Geometrically, the idea is to study analytic functions on deleted neighborhoods within formal schemes.  This necessitates working with not just a localization of a complete ring, but also remembering the complete ring from which it stems, together with both of their defining direct and inverse systems.  Along with more minor examples, we prove a Lubin--Tate result for $L_{K(t)} E_n^0$ in this extended setting:
\begin{maintheorem}[{Proven below as Theorem~\ref{thm:LubinTatePair}}]
Fix a finite ground field $k$ of positive characteristic $p$ and a formal group law $\Gamma$ of finite height $n$ over $k$.  Select an infinitesimal deformation $X_0$ of $\Spec k$ and a ``punctured affine formal scheme'' $X_1 = \Spp R$ over $X_0$.  Then, for each $t \le n$, there is an equivalence between commuting diagrams of the shape
\begin{center}
\begin{tikzcd}
\Spec k \arrow{r} \arrow[double,-]{d} & X_0 \arrow{d} & X_1 \arrow{l} \arrow{d} \\
\Spec k \arrow{r} & \Spf E_n^0 & \Spp L_{K(t)} E_n^0 \arrow{l}
\end{tikzcd}
\end{center}
and formal group laws $\G_0$ over $X_0$ which deform $\Gamma$ and whose pullback to $X_1$ is of height $t$.
\end{maintheorem}

\subsection{Acknowledgements}

There are a number of people to thank for their contributions to this paper: Olga Stroilova contributed substantially to the initial ideas; Dustin Clausen very helpfully suggested that we look to the algebraic geometry of local fields of large dimension for existing ideas; Jared Weinstein gave us a useful guided tour of many existing extensions of the theory of formal schemes; Charles Rezk is responsible for all three of our interests in the cohomology theory $L_{K(t)} E_n$; Neil Strickland has been persistently inspiring in his use of algebraic geometry to explain phenomena in algebraic topology; Adam Prescott suggested the name ``pipe rings'' over a multitude of worse ideas we'd previously tried; Justin Noel read an early version of this paper and caught a number of crucial errors; and the generosities of Haynes Miller, Peter Teichner, and Constantin Teleman all came together to make possible the visits of the first two authors to MIT, where all the work on this paper happened.  The first author was supported in part by NSF GRFP grant DGE-1106400; the first and second authors were both supported in part by UC Berkeley's geometry and topology RTG grant, which is part of DMS-0838703; and the third author was partially supported by NSF grant DMS-0943787.  Finally, all three authors thoroughly enjoyed the comfort and hospitality of the Muddy Charles.

\section{Continuity after Kato}

The goal of this section is to construct and study a category of ``pipe rings'' satisfying the following desiderata:
\begin{enumerate}
\item The usual category of profinite rings and continuous maps contributes a full subcategory of pipe rings.  That is to say: pipe rings will carry a ``topology'' in general, which reduces to an honest adic topology in classical cases.
\item The localization morphism $\pi_0 E_n \to \pi_0 L_{K(t)} E_n$ belongs to this category, as do the more general iterated localization morphisms \[\pi_0 E_n \to \pi_0 L_{K(t)} E_n \to \pi_0 L_{K(t')} L_{K(t)} E_n \to \cdots.\]
\end{enumerate}
Kato has constructed such a category of rings with these properties in an effort to understand the continuity properties of local fields of large dimension~\cite{Kato}.  Using our terminology, we recall his construction and point out various properties and constructions which we now need but which he has not already recorded.

\begin{definition}
Let $\CatOf{Pipes}_{-1}$ denote the category of finite sets and $\CatOf{Pipes}_0$ the category of profinite sets, which we refer to as \textit{$(-1)$-pipes} and \textit{$0$-pipes} respectively.  For $n \ge 1$, we inductively define the category $\CatOf{Pipes}_n$ of \textit{$n$-pipes} by the formula \[\CatOf{Pipes}_n = \Pro(\Ind \CatOf{Pipes}_{n-1}),\] and we refer to $n$ as the \textit{length}.  (Note that $\CatOf{Pipes}_0$ is \emph{not} equal to $\Pro(\Ind \CatOf{Pipes_{-1}})$.)  There are fully faithful product-respecting inclusions $\CatOf{Pipes}_{n-1} \to \CatOf{Pipes}_n$, given by sending an object of $\CatOf{Pipes}_{n-1}$ to its associated constant system in $\Pro(\Ind \CatOf{Pipes}_{n-1}) = \CatOf{Pipes}_n$.  We denote the sequential colimit along these inclusions by $\CatOf{Pipes}_\infty$, and its objects are referred to simply as \textit{pipes}.\footnote{The word ``pipe'' here is an acronym for Pro-Ind-Pro Ensemble.}
\end{definition}

\begin{definition}
The categories $\CatOf{Pipes}_n$ and $\Ind \CatOf{Pipes}_{n-1}$ admit finite products, so we may speak of their associated categories of ring objects $\CatOf{PipeRings}_n = \Rings(\CatOf{Pipes}_n)$ and $\Rings(\Ind \CatOf{Pipes}_{n-1})$.  As shorthand, we refer to an object of $\CatOf{PipeRings}_n$ as an \textit{$n$-pipe ring}.  The inclusions $\CatOf{Pipes}_{n-1} \to \CatOf{Pipes}_n$ and $\Ind \CatOf{Pipes}_{n-1} \to \Ind \CatOf{Pipes}_n$ are product-preserving, so induce inclusions of categories of ring objects.  In the aggregate we refer to these as \textit{pipe rings}.
\end{definition}

\subsection{An embedding condition}

For convenience, we will write Hom-sets as $[-,-]$, regardless of the ambient category.  The constant system at a singleton set gives a terminal object $1 \in \CatOf{Pipes}_\infty$, and we define a functor $\CatOf{Pipes}_\infty \to \CatOf{Sets}$ by \[S \mapsto [1, S] =: \real{S},\] called \textit{(set-theoretic) realization}.  

This definition can also be given inductively.  On the one hand, if $S \in \CatOf{Pipes}_n$ is given as a pro-system $\{S_\alpha\}_\alpha$, then (considering each $S_\alpha \in \CatOf{Pipes}_n$) we have that
\[ \real{S} = [1,S] = [1,\lim_\alpha S_\alpha] = \lim_\alpha [1,S_\alpha] = \lim_\alpha \real{S_\alpha} . \]
On the other hand, if additionally each $S_\alpha$ is given as an ind-system $\{(S_\alpha)_\beta\}_\beta$ of $(n-1)$-pipes, then we analogously have
\[ \real{S_\alpha} = [1 , S_\alpha] = [1,\colim_\beta (S_\alpha)_\beta] = \colim_\beta [1,(S_\alpha)_\beta] = \colim_\beta \real{(S_\alpha)_\beta} \]
since $1$, as a finite set, is a small object.  Thus, taking the realization is exactly the iterative process of taking limits of pro-systems and colimits of ind-systems.  This implies that realization commutes with finite limits and in particular with finite products, so it induces a functor $\CatOf{PipeRings}_\infty \to \CatOf{Rings}$.  

This functor should be thought of as ``forgetful'', in the sense of sending a topologized ring to its underlying ring.  As one expects, this functor is not injective on objects; just as a given set can support many topologies, so do sets support many pipe structures.  However, as we have set up the category $\CatOf{Pipes}_\infty$, there is no reason to even expect realization to be \emph{faithful}.  This is a more serious problem.

\begin{definition}
To control this aspect of realization, we make two further inductive definitions:
\begin{enumerate}[label=(\alph*)]
\item Every $(-1)$-pipe and $0$-pipe is called \textit{fine}.  An $n$-pipe $Y$ is called \textit{fine} if it can be expressed as a pro-system $\{Y_\alpha\}_\alpha$ of ind-systems $\{(Y_\alpha)_\beta\}_\beta$ such that:
\begin{itemize}
\item Each object of $(Y_\alpha)_\beta$ is a fine $(n-1)$-pipe.
\item The induced map $\real{(Y_\alpha)_\beta} \to \real{Y_\alpha}$ is injective for every choice of $\alpha$ and $\beta$.
\end{itemize}
\item Every $(-1)$-pipe is called \textit{cofine}.  A $0$-pipe $X$ is called \textit{cofine} if it can be expressed as a pro-system $\{X_\lambda\}_\lambda$ of finite sets such that the induced map $\real{X} \to \real{X_\lambda}$ is surjective for each choice of $\lambda$.  Generally, an $n$-pipe $X$ is called \textit{cofine} if it can be expressed as an ind-system $\{X_\lambda\}_\lambda$ of pro-systems $\{(X_\lambda)_\mu\}_\mu$ such that:
\begin{itemize}
\item Each object of $(X_\lambda)_\mu$ is a cofine $(n-1)$-pipe.
\item The induced map $\real X \to \real{X_\lambda}$ is surjective for every choice of $\lambda$.
\end{itemize}
\end{enumerate}
\end{definition}

\begin{lemma}
The properties of being fine and cofine are both preserved by the inclusions $\CatOf{Pipes}_{n-1} \to \CatOf{Pipes}_n$. \qed
\end{lemma}

Altogether, the point is the following lemma:
\begin{lemma}
The realization functor is faithful when applied to maps $X \to Y$ for which $X$ is cofine and $Y$ is fine.
\end{lemma}
\begin{proof}
There exists an integer $n$ for which both $X$ and $Y$ are realized as $n$-pipes and the map $X \to Y$ is exhibited as a map of $n$-pipes.  It thus suffices to check the situation where the source and target have the same length.

Firstly, that realization is faithful on such maps of $(-1)$- and $0$-pipes is trivial.  So, suppose that realization is faithful when applied to such maps of $(n-1)$-pipes.  We compute the action of realization of such a map of $n$-pipes as follows:
\begin{align*}
[X, Y] & = \lim_\alpha \colim_\lambda [X_\lambda, Y_\alpha] = \lim_\alpha \colim_\lambda \lim_\nu \colim_\beta [(X_\lambda)_\nu, (Y_\alpha)_\beta] \\
& \subset \lim_\alpha \colim_\lambda \lim_\nu \colim_\beta [\real{(X_\lambda)_\nu}, \real{(Y_\alpha)_\beta}] & \hbox{(inductive assumption)} \\
& \subset \lim_\alpha \colim_\lambda \lim_\nu [\real{(X_\lambda)_\nu}, \real{Y_\alpha}] & \hbox{(fineness of $Y$)}\\
& = \lim_\alpha \colim_\lambda [ \colim_\nu \real{(X_\lambda)_\nu} , \real{Y_\alpha} ] \\
& = \lim_\alpha \colim_\lambda [\real{X_\lambda}, \real{Y_\alpha}] \\
& \subset \lim_\alpha [\real{X}, \real{Y_\alpha}] & \hbox{(cofineness of $X$)} \\
& = [\real{X}, \real{Y}]. & & \qedhere
\end{align*}
\end{proof}

This lemma demonstrates that $\CatOf{PipeRings}_\infty$ does indeed satisfy our first desideratum.  More to the point, it illustrates the analogy with the classical situation of topologized sets, and indicates that we might consider an $n$-pipe as determining a ``generalized topology'' on its set-theoretic realization.

These conditions are reminiscent of an analogous situation in abstract homotopy theory, where a model structure on a category determines which are the ``right'' objects to map to and from.  Unfortunately, despite substantial effort, we have been unable to precisely pin down this analogy.  Instead, let us record the following conjecture:
\begin{conjecture}[Pipe dream]\label{pipe dream}
For every pipe $X$, there exists an initial cofine pipe $X^c$ over $X$ such that the map $X^c \to X$ induces an isomorphism $\real{X^c} \to \real{X}$.  Dually, for every pipe $Y$, there exists a terminal fine pipe $Y^f$ under $Y$ such that the map $Y \to Y^f$ induces an isomorphism $\real Y \to \real{Y^f}$.  Finally, there is a class of weak equivalences $W$ for which these compute the derived mapping space: \[ \Hom_{\CatOf{Pipes}_\infty[W^{-1}]}(X, Y) \cong \Hom_{\CatOf{Pipes}_\infty}(X^c, Y^f).\]
\end{conjecture}

\begin{remark}
A cofineification functor for $\CatOf{Pipes}_0$ can be constructed as follows: let $X$ be a profinite set, presented as a pro-system $\{X_\alpha\}_\alpha$ of finite sets.  Fixing any particular $X_\alpha$ in the system, we can consider the objects $X_\beta \to X_\alpha$ over it, each of which has a corresponding image factorization \[X_\beta \to \operatorname{im}(X_\beta \to X_\alpha) \to X_\alpha.\]  If $X_\gamma$ is a yet ``further'' object over $X_\alpha$, in the sense that there is a factorization $X_\gamma \to X_\beta \to X_\alpha$ of the map $X_\gamma \to X_\alpha$, then this induces an inclusion $\operatorname{im}(X_\gamma \to X_\alpha) \to \operatorname{im}(X_\beta \to X_\alpha)$, and altogether gives a pro-system of images over $X_\alpha$.  Since each map in this system is a monomorphism and $X_\alpha$ is a finite set, this pro-system is necessarily equivalent to a constant pro-system on a finite subset $\operatorname{evim}(X_\alpha)$, the ``eventual image''.  Varying $\alpha$, the resulting system of eventual images forms an initial cofine $0$-pipe over $X$.  In the general case of $\CatOf{Pipes}_n$, however, we cannot guarantee that the pro-system of images can be replaced by a constant ``eventual image'' system, and it seems that the core feat of any construction of a cofineification functor on $n$-pipes would be to work around this fact.
\end{remark}

\subsection{Closed ideals and formal geometry}

In algebraic geometry, one studies rings through their associated categories of modules, and in particular through their ideals.  From a categorical perspective, ideals should be thought of as kernels of ring maps; then, just as the restriction of continuity on maps of topologized rings gives rise to the notion of a closed ideal, so should our ``generalized topologies'' determine the correct notion of an ideal in our setting.  In this section, we make this notion precise and study some of its basic features, although we explore little beyond what we will need in the remainder of the paper.  Throughout, we will make quiet use of basic facts about pro-categories, an excellent reference for which is Isaksen's paper~\cite{Isaksen}.

Associated to an $n$-pipe ring $R$, we produce an associated theory of modules by studying $n$-pipe abelian groups $M$ with action maps $R \times M \to M$.  As usual, an ideal $I$ is then such an $R$-module equipped with a monomorphism $I \to R$ of $R$-modules.

\begin{definition}
Let $R$ be an $n$-pipe ring.  An ideal $\real{I} \subset \real{R}$ is said to be \textit{closed} when it is the image under realization of a sequence \[I \xrightarrow{\operatorname{ker}} R \to S.\]
\end{definition}

We can further refine this definition.  Notice that there are $(n+1)$ naturally occurring inclusions $i_m$ of $\CatOf{Pipes}_n$ into $\CatOf{Pipes}_{n+1}$.  For instance, in the case of $n = 0$, given a profinite set $\{X_\alpha\}_\alpha$ one can specify two new pro-ind-profinite sets using the following pair of formulas:
\begin{align*}
\{\{(i_0 X)_\beta\}_\gamma\}_\delta & = X_\delta, &
\{\{(i_1 X)_\beta\}_\gamma\}_\delta & = X_\beta.
\end{align*}
That is, to construct $i_0 X$ we consider the pro-system $\{X_\alpha\}_\alpha$ as a constant system in $\Ind \CatOf{Pipes}_0$, then consider that as a constant system in $\Pro(\Ind \CatOf{Pipes}_0) = \CatOf{Pipes}_1$.  On the other hand, to construct $i_1X$ we consider each finite set $X_\alpha$ as a constant system in $\Pro \CatOf{Pipes}_{-1} = \CatOf{Pipes}_0$, then consider each of those as a constant system in $\Ind \CatOf{Pipes}_0$, and then finally piece these objects together using the original structure maps of $X$ to get a system in $\Pro(\Ind \CatOf{Pipes}_0) = \CatOf{Pipes}_1$.  The standard inclusion $\CatOf{Pipes}_0 \to \CatOf{Pipes}_1$ used in the sequential colimit defining $\CatOf{Pipes}_\infty$ is $i_0$.  The general pattern is similar.

\begin{definition}
A closed ideal $I$ occurring as the kernel of a map $R \to S$ of pipe rings is said to be \textit{stationary of level $m$} when the target $n$-pipe ring $S$ can be written as $i_m$ applied to an $(n-1)$-pipe ring.  More coarsely, we say $I$ is a \textit{$j$-ideal} if it appears as the kernel of a map to an $j$-pipe ring (i.e., it is stationary at all levels between $(j+1)$ and $n$).
\end{definition}

Stationary ideals are an expected phenomenon, given the analogy with classical topology: just as a map of spaces $[0, 1) \to S^1$ can be a continuous bijection but not a homeomorphism, we can produce maps $X \to Y$ of pipes which are not themselves isomorphisms but which induce bijections $\real X \to \real Y$ on realizations.

\begin{definition}
A \textit{clogged pipe} is a pipe $Y$ which admits a map $X \to Y$ which is not an isomorphism but which realizes to an isomorphism $\real X \to \real Y$.
\end{definition}

\begin{example}
Examples of clogged pipes are rife in pipe rings, owing to the inclusions $i_m: \CatOf{Pipes}_{n-1} \to \CatOf{Pipes}_n$.  For instance, let $R\llbracket x \rrbracket$ denote the $0$-pipe ring \[R\llbracket x \rrbracket := \{\cdots \to R[x] / x^3 \to R[x] / x^2 \to R\},\] and consider the $1$-pipe rings
\begin{align*}
Q_0 & = i_0(R\powser{x}) = \operatorname{const}(R\llbracket x \rrbracket), \\
Q_1 & = i_1(R\powser{x}) = \{\cdots \operatorname{const}(R[x] / x^3) \to \operatorname{const}(R[x] / x^2) \to \operatorname{const}(R)\}.
\end{align*}
There is a map $Q_0 \to Q_1$ given by a levelwise quotient $R \llbracket x \rrbracket \to R[x] / x^n$, and while this map realizes to the identity on set-theoretic realizations, no inverse map exists.  Instead, its kernel is given by the nonzero pro-ideal $\{(x^n)\}_{n \ge 1}$, which (given the construction of $Q_1$) we see is stationary of level $1$.
\end{example}

\begin{remark}
Clogs also point out an interesting feature of the co/fine properties defined in the previous subsection.  The only cofine ideal which realizes to the zero ideal under set-theoretic realization is isomorphic to the zero ideal itself.  Hence, no ideal representing a clogged pipe can ever be cofine.  From the perspective of Conjecture~\ref{pipe dream}, this implies that it cannot suffice to simply define the class $W$ of weak equivalences to be those maps which set-theoretically realize to isomorphisms.
\end{remark}

Central to basic algebraic geometry are the notions of being complete and of being local.  We now give analogous definitions in our setting.

\begin{definition}
Recall that $\Pro \CatOf{C}$ is sequentially complete for any base category $\CatOf{C}$.  An $n$-pipe ring $R$ is said to be \textit{$n$-complete with respect to the ideal $I$} (or just \textit{complete}) if $R$ can be expressed as the limit \[R = \lim \left( \cdots \to R / I^2 \to R / I \right)\] in the category of $n$-pipe rings.
\end{definition}

\begin{definition}
Let $R$ be an $n$-pipe ring.  Then for any $x \in \real{R}$ we have a multiplication map
\[ R = 1 \times R \xrightarrow{x \times \id} R \times R \xrightarrow{\mu} R . \]
From this, we can construct $x^{-1}R$ as the ind-object
\[ x^{-1}R = \left\{ R \xrightarrow{x \cdot -} R \xrightarrow{x \cdot -} R \xrightarrow{x \cdot - } \cdots \right\} . \]
More generally, for any multiplicatively closed subset $T \subset \real{R}$ we can analogously form $T^{-1}R$.  If the natural map $R \to T^{-1} R$ is an isomorphism, we say that $R$ is \textit{local} with respect to $T$.
\end{definition}

\begin{proposition}
Suppose $T \subset \real{R}$ contains no zerodivisors.  Then $T^{-1}R$ is fine if $R$ is, and $T^{-1}R$ is cofine if $R$ is.
\end{proposition}
\begin{proof}
The maps in the defining ind-system for $T^{-1}R$ induce injections on set-theoretic realizations since we are assuming all the elements of $T$ are non-zerodivisors.  On the other hand, inducing up to a pro-system, the cofineness condition is vacuous.
\end{proof}

\begin{remark}
We are deliberately careful to use the notation $x^{-1} R$ in place of $R[x^{-1}]$, to avoid confusing this operation with the quotient construction $R[y]/(xy-1)$.  This latter operation is poorly behaved in pipe rings --- for instance, the finiteness restriction of $\CatOf{Pipes}_{-1}$ can prevent it from existing at all.
\end{remark}

It's worth pointing out that it's obviously difficult to get a handle on the closed ideals of a pipe ring in general.  In the profinite case of $\CatOf{Pipes}_0$, there is a comparison with Stone spaces~\cite[VI.2.3]{Johnstone} that yields an equivalence $\Pro(\CatOf{FiniteRings}) \simeq \CatOf{PipeRings}_0$, but in general we do not expect such an equivalence.  However, our examples of greatest interest do come with this extra structure, which makes this task slightly easier.

First observe that there is a functor $\Pro(\Rings \CatOf{C}) \to \Rings(\Pro \CatOf{C})$.  For example, a pro-ring object $\{R_\alpha\}_\alpha$ inherits a multiplication map $\{R_\alpha\}_\alpha \times \{R_\alpha\}_\alpha = \{R_\alpha \times R_\beta\}_{\alpha,\beta} \to \{R_\alpha\}_\alpha$ by restricting to the cofinal diagonal subsystem $\{R_\alpha \times R_\alpha\}_\alpha \hookrightarrow \{R_\alpha \times R_\beta\}_{\alpha,\beta}$ and using the multiplication maps on the individual rings $R_\alpha$.

\begin{definition}
Fix a ring $R$.  A \textit{rigid $0$-pipe ring} is a $0$-pipe ring which can be expressed as a pro-system of finite $R$-algebras and $R$-algebra morphisms (where by ``finite'' we mean ``finite as a set'').  A \textit{rigid $n$-pipe ring} is an $n$-pipe ring which can be written  as a diagram on the product of projective indexing sets $P_i$ and inductive indexing sets $I_j$:
\[P_0 \times I_1 \times P_1 \times \cdots \times I_n \times P_n \to \CatOf{AbelianGroups}\] in such a way that the following conditions are satisfied:
\begin{itemize}
\item Each object is a finite ring.
\item Each projective morphism is a morphism of rings.
\item Each inductive morphism is a morphism of $R$-modules.
\item For each choice of sequence $(i_{j+1}, p_{j+1}, i_{j+2}, \ldots, i_n, p_n)$ of terminal indices, the restricted system \[P_0 \times I_1 \times \cdots \times P_j \times \{i_{j+1}\} \times \{p_{j+1}\} \times \cdots \times \{p_n\} \to \CatOf{AbelianGroups}\] determines an object in $\CatOf{PipeRings}_j$ via the functor described above.
\end{itemize}
\end{definition}

\begin{proposition}\label{prop:RigidQuotient}
Given a rigid $n$-pipe ring $S$ whose constituent $0$-pipe rings are all quotients of a fixed ring $R$, an $R$-ideal $I$ begets a closed $S$-ideal by tensoring the defining diagram for $S$ with $I$.
\end{proposition}
\begin{proof}
The only unclear point to check is that if $A$ is a finite $R$-algebra in the rigid system expressing $S$, then $A \otimes_R I$ is again finite (i.e., a $(-1)$-pipe).  Because $I$ is an ideal of $S$, this module is equivalent to $IA \subset A$, and since $A$ is finite, so is $IA$.
\end{proof}

Oftentimes --- including in the statement of our main theorem --- it will be useful to track the stages of the construction of a pipe ring in a different way.

\begin{definition}
We define \textit{staged pipe rings} to be the full subcategory $\CatOf{StagedRings} \subset \textup{Fun}(\mathbb{N},\CatOf{PipeRings}_\infty)$ of objects $R_*$ such that $R_n$ is an $n$-pipe ring for all $n$.
\end{definition}

\begin{remark}
Relatedly, there is a fully faithful embedding $\CatOf{Sets} \to \CatOf{Pipes}_1$, since every set can be canonically identified with its ind-system of finite subsets.  In turn, this gives a full embedding $\CatOf{Rings} \to \CatOf{PipeRings}_1$.  It's unclear to the authors whether this is an indication that the category $\CatOf{Pipes}_\infty$ is still far too large to be the ``right'' category or if it is simply an unavoidable feature of working with ind-systems.
\end{remark}

\subsection{Basic examples and the portrait of a pipe ring}

The point of algebraic geometry in general and schemes in particular is to provide very literally an interface between commutative algebra and geometry.  In particular, a scheme comes with a (complicated and difficult) recipe for drawing a picture of it, which is extremely useful for building geometric intuition about the behavior of algebraic objects when it can even partially be carried out.  In this subsection, we produce a construction that provides similar information, and we use it to describe a handful of examples which will be useful later.

\begin{definition}
We define $\CatOf{FilteredSpaces}$ to be the category of sequences \[X_{-1} \to X_0 \to X_1 \to X_2 \to \cdots\] of topological spaces with each map the inclusion of a closed subspace.  A morphism in this category is a natural transformation of such diagrams (without further restriction).
\end{definition}

\begin{definition}
We define a functor \[\portrait: \CatOf{PipeRings}_\infty \to \CatOf{FilteredSpaces},\]
called the \emph{portrait} of a pipe ring.  The filtration level $\portrait(R)_n$ is given by the set of those isomorphism classes of closed $m$-ideals of $R$ which are of length $m \le n$.  Each of these spaces is topologized by setting the closure of a point $I \in \portrait(R)$ to be the collection of ideals which contain it, then closing under finite union.  That is, the closure of a point gives a sub-filtered space of the portrait $\portrait(R)$.
\end{definition}

Before we proceed to the examples, two brief remarks are in order.

\begin{remark}
Our definition of the portrait functor lacks an accompanying locally ringed structure sheaf.  We have a definition for the localization of a pipe ring, and so could make an ostensible guess as to a definition of the structure sheaf by mimicking the classical setting, but it doesn't appear to lead where we would like.  So, we omit it, with some disappointment.
\end{remark}

\begin{remark}
The reader should also be warned that, even were we equipped with a good notion of structure sheaf, the gluing of affine portraits to form non-affine portraits also seems to be poorly behaved.  The most basic request of a successful such theory would be to form a formal analogue of $\mathbb{P}^1$, by joining two copies of formal affine space $k\powser{x}$ and $k\powser{y}$ along the ``punctured'' formal affine space $k\(z\)$, using the gluing data of $x \mapsto z$ and $y \mapsto z^{-1}$.  However, this second map is obviously malformed, owing to the fact that $z$ and $z^{-1}$ play different privileged roles in the formal Laurent series ring $k\(z\)$.  We do not pretend that these portraits are anything like a step on the road to a theory of gluing for formal schemes.
\end{remark}

\subsubsection{Example: $k\powser{x} \to k\(x\)$}

The basic recipe is to draw a point for each power of each closed prime $n$-ideal, to label the points by the index $n$, and to topologize the space by defining the closure of a point to be the collection of ideals which contain it.

As a first example, let's begin with $k\powser{x}$, where $k$ is a finite field.  This $0$-pipe ring has a single closed point corresponding to the kernel of the map $k\powser{x} \to k$;  we represent it by a large dot.  Then, there is an ascending sequence of closed nilpotent thickenings of this closed point given by maps
\[
k\powser{x} \to k[x]/x^n
\]
onto nilpotent extensions of $k$.  Each of these targets is a finite ring, and so the corresponding ideal is a $(-1)$-ideal.  Together, we imagine these as an infinite sequence of smaller dots, topologized as indicated.  Finally, there is a single $0$-ideal, corresponding to the identity map and the ideal $(0)$, which we draw as a fuzzy dot at the limit of the smaller ones.  As the generic point, its closure is the entire space.

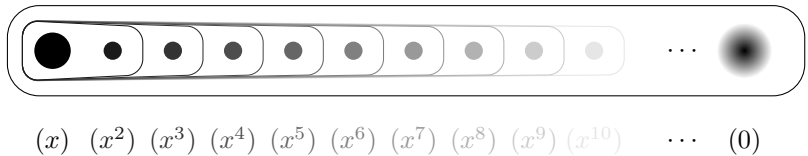
\begin{figure}[htc]
\begin{tikzpicture}[scale=4]
\fill [color=black] (0, 0) circle (0.4ex);
\fill [color=black!90] (0:0.2) circle (0.2ex);
\fill [color=black!80] (0:0.4) circle (0.2ex);
\fill [color=black!70] (0:0.6) circle (0.2ex);
\fill [color=black!60] (0:0.8) circle (0.2ex);
\fill [color=black!50] (0:1.0) circle (0.2ex);
\fill [color=black!40] (0:1.2) circle (0.2ex);
\fill [color=black!30] (0:1.4) circle (0.2ex);
\fill [color=black!20] (0:1.6) circle (0.2ex);
\fill [color=black!10] (0:1.8) circle (0.2ex);
\fill [inner color=black, color=white] (0:2.3) circle (0.6ex);

\draw [color=black!10,rounded corners=6pt] (-0.1,-0.1) -- (-0.1,0.1) -- (1.90,0.08) -- (1.90,-0.08) -- cycle;
\draw [color=black!20,rounded corners=6pt] (-0.1,-0.1) -- (-0.1,0.1) -- (1.70,0.08) -- (1.70,-0.08) -- cycle;
\draw [color=black!30,rounded corners=6pt] (-0.1,-0.1) -- (-0.1,0.1) -- (1.50,0.08) -- (1.50,-0.08) -- cycle;
\draw [color=black!40,rounded corners=6pt] (-0.1,-0.1) -- (-0.1,0.1) -- (1.30,0.08) -- (1.30,-0.08) -- cycle;
\draw [color=black!50,rounded corners=6pt] (-0.1,-0.1) -- (-0.1,0.1) -- (1.10,0.08) -- (1.10,-0.08) -- cycle;
\draw [color=black!60,rounded corners=6pt] (-0.1,-0.1) -- (-0.1,0.1) -- (0.90,0.08) -- (0.90,-0.08) -- cycle;
\draw [color=black!70,rounded corners=6pt] (-0.1,-0.1) -- (-0.1,0.1) -- (0.70,0.08) -- (0.70,-0.08) -- cycle;
\draw [color=black!80,rounded corners=6pt] (-0.1,-0.1) -- (-0.1,0.1) -- (0.50,0.08) -- (0.50,-0.08) -- cycle;
\draw [color=black!90,rounded corners=6pt] (-0.1,-0.1) -- (-0.1,0.1) -- (0.30,0.08) -- (0.30,-0.08) -- cycle;
\draw [color=black,rounded corners=12pt] (-0.15,-0.15) -- (-0.15,0.15) -- (2.5,0.15) -- (2.5,-0.15) -- cycle;

\draw (0,-0.3) node{$(x)$};
\draw[color=black!90] (0.2,-0.3) node{$(x^2)$};
\draw[color=black!80] (0.4,-0.3) node{$(x^3)$};
\draw[color=black!70] (0.6,-0.3) node{$(x^4)$};
\draw[color=black!60] (0.8,-0.3) node{$(x^5)$};
\draw[color=black!50] (1.0,-0.3) node{$(x^6)$};
\draw[color=black!40] (1.2,-0.3) node{$(x^7)$};
\draw[color=black!30] (1.4,-0.3) node{$(x^8)$};
\draw[color=black!20] (1.6,-0.3) node{$(x^9)$};
\draw[color=black!10] (1.8,-0.3) node{$(x^{10})$};
\draw (2.1,0.0) node{$\cdots$};
\draw (2.1,-0.3) node{$\cdots$};
\draw (2.3,-0.3) node{$(0)$};
\end{tikzpicture}
\caption{Portrait of $k\powser{x}$. The solid black dots correspond to $(-1)$-ideals, the fuzzy dot to a $0$-ideal.}
\end{figure}

Continuous maps $k\powser{x} \to R$ pick out elements which can be either nilpotent or merely topologically nilpotent.  The difference is detected by whether the corresponding ideal is a $(-1)$- or $0$-ideal --- an observation that will be useful as we now study the ind-profinite ring $x^{-1}k\powser{x}$.  This ind-object is defined by iterating the multiplication-by-$x$ map, which has a visible action on the ideals: it sends $(x^n)$ to $(x^{n+1})$ and $(0)$ to $(0)$.  Taking the colimit, we see that the subset on which $x^{-1}k\powser{x}$ is supported is $(0)$, and the portrait for $k\(x\)$ reflects this.  We labeled the point corresponding to $(0)$ with a $1$, to reflect that it's the kernel of a morphism with target in $\CatOf{PipeRings}_1$.  It is also immediately apparent from functoriality of the portrait construction that a morphism of pipe rings $k\(x\) \to R$ cannot send $x$ to a nilpotent element, as the inclusion of the portrait for $k\powser{x} / x^n$ into the one for $k\powser{x}$ evidently does not factor through the one for $k\(x\)$.

\begin{figure}[htc]
\begin{tikzpicture}[scale=4]
\fill [inner color=black, color=white] (0:2.3) circle (0.8ex);
\draw [fill=white,color=white] (0:2.3) circle (0.3ex);
\draw (0:2.3) node{$1$};

\draw [color=black,rounded corners=12pt] (-0.15,-0.15) -- (-0.15,0.15) -- (2.5,0.15) -- (2.5,-0.15) -- cycle;

\draw (2.3,-0.3) node{$(0)$};
\end{tikzpicture}
\caption{Portrait of $x^{-1} k\powser{x} = k\(x\)$.}
\end{figure}
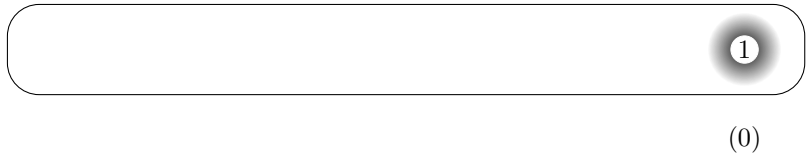

\begin{remark}
Having worked this example, one can explore the obvious definition of the locally ringed structure sheaf that could be attached to this space.  The local ring assigned to the point $(x)$ is the ring $k\powser{x}$ itself, but every other point is assigned the field $k\(x\)$, making it too coarse an invariant to effectively distinguish between the other points we've added.  It's unclear how to correct that construction to something more sensitive.
\end{remark}

\begin{remark}
This same method produces portraits for $\Z_p$ and $\mathbb{Q}_p$, isomorphic to the ones given above, though it's helpful when working to keep an ``arithmetic direction'' distinct from the rest.
\end{remark}

\begin{remark}
Throughout, we have taken $k$ to be a finite field, and this is an inescapable feature of our set-up.  Similar to $k\(x\)$ in this example, more general local fields can be instantiated as $n$-pipe rings for $n > -1$; see work of Kato~\cite[Section 1.2]{Kato}.
\end{remark}

\subsubsection{Example: $k\powser{x,y} \to y^{-1} k\powser{x,y} \to (y^{-1}k\powser{x,y})_{(x)}^{\wedge}$}

This is the main example.  As before, we proceed in stages, beginning with the $0$-pipe ring $k\powser{x,y}$.  This has both closed $(-1)$-ideals and $0$-ideals: examples of $(-1)$-ideals include $(x, y)$, $(x^i, y^j)$ for $i, j > 0$, and $(x^2, xy, y^2)$, whereas examples of $0$-ideals include $(x)$, $(x+y)$, $(x^n)$ for $n > 0$, and $(xy)$.  To analyze the topology, notice first that the closure of an $n$-ideal may only contain $m$-ideals for $m \le n$.  For instance, $0$-ideal $(x^2)$ is contained in both the $0$-ideal $(x)$ and the $(-1)$-ideal $(x^2, y)$.

To give just an abbreviation of the full portrait, we draw only the points for powers of prime ideals, as in Figure~\ref{fig:Pipe3}.
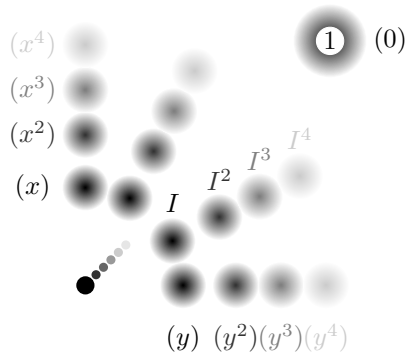
\begin{figure}
\begin{tikzpicture}[scale=2]
\fill [color=black] (0, 0) circle (0.4ex);
\fill [color=black!80] (45:0.10) circle(0.2ex);
\fill [color=black!60] (45:0.17) circle(0.2ex);
\fill [color=black!40] (45:0.24) circle(0.2ex);
\fill [color=black!20] (45:0.31) circle(0.2ex);
\fill [color=black!10] (45:0.38) circle(0.2ex);

\fill [inner color=black, color=white] (0:0.65) circle (1ex);
\fill [inner color=black, color=white] (27:0.65) circle (1ex);
\fill [inner color=black, color=white] (63:0.65) circle (1ex);
\fill [inner color=black, color=white] (90:0.65) circle (1ex);

\fill [inner color=black!80, color=white] (0:1) circle (1ex);
\fill [inner color=black!80, color=white] (27:1) circle (1ex);
\fill [inner color=black!80, color=white] (63:1) circle (1ex);
\fill [inner color=black!80, color=white] (90:1) circle (1ex);

\fill [inner color=black!50, color=white] (0:1.3) circle (1ex);
\fill [inner color=black!50, color=white] (27:1.3) circle (1ex);
\fill [inner color=black!50, color=white] (63:1.3) circle (1ex);
\fill [inner color=black!50, color=white] (90:1.3) circle (1ex);

\fill [inner color=black!20, color=white] (0:1.6) circle (1ex);
\fill [inner color=black!20, color=white] (27:1.6) circle (1ex);
\fill [inner color=black!20, color=white] (63:1.6) circle (1ex);
\fill [inner color=black!20, color=white] (90:1.6) circle (1ex);

\draw (-0.35,0.65) node{$(x)$};
\draw[color=black!80] (-0.35,1) node{$(x^2)$};
\draw[color=black!50] (-0.35,1.3) node{$(x^3)$};
\draw[color=black!20] (-0.35,1.6) node{$(x^4)$};
\draw (0.65,-0.35) node{$(y)$};
\draw[color=black!80] (1,-0.35) node{$(y^2)$};
\draw[color=black!50] (1.3,-0.35) node{$(y^3)$};
\draw[color=black!20] (1.6,-0.35) node{$(y^4)$};
\draw[color=black] (27:0.65) ++ (0,0.25) node{$I$};
\draw[color=black!80] (27:1.0) ++ (0,0.25) node{$I^2$};
\draw[color=black!50] (27:1.3) ++ (0,0.25) node{$I^3$};
\draw[color=black!20] (27:1.6) ++ (0,0.25) node{$I^4$};

\fill [inner color=black, color=white] (45:2.3) circle (1.6ex);
\draw [fill=white,color=white] (45:2.3) circle (0.6ex);
\draw (45:2.3) node{$1$};

\draw (45:2.3) ++ (0.4,0) node{$(0)$};
\end{tikzpicture}
\caption{Portrait of $k\powser{x,y}$.  Here, $I$ is a closed prime $0$-ideal.  The horizontal and vertical axes are labeled by $(y)$ and $(x)$ respectively, as these are the subschemes selected by these ideals. \label{fig:Pipe3}}
\end{figure}
The closed point sits at the bottom left, together with its string of powers.  The prime $0$-ideals are also drawn in, with their powers marching away along lines of their own.  The ideal $(0)$ is the generic point: it is not contained in the closure of any other point, and its closure is the entire space.  This abbreviation is reasonable because the portrait is topologized: we are working in a UFD, so an arbitrary closed ideal will be uniquely characterized by a finite collection of these points.  We provide pictures of the closures of the non-prime ideals $(y^3)$ and $(x^3 (x+y)^2 y)$ in Figure~\ref{fig:Closures}.

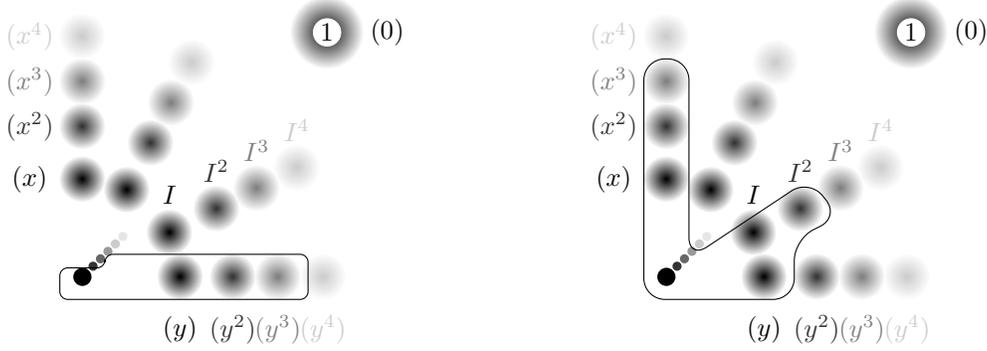
\begin{figure}
\begin{tikzpicture}[scale=2]
\fill [color=black] (0, 0) circle (0.4ex);
\fill [color=black!80] (45:0.10) circle(0.2ex);
\fill [color=black!60] (45:0.17) circle(0.2ex);
\fill [color=black!40] (45:0.24) circle(0.2ex);
\fill [color=black!20] (45:0.31) circle(0.2ex);
\fill [color=black!10] (45:0.38) circle(0.2ex);

\fill [inner color=black, color=white] (0:0.65) circle (1ex);
\fill [inner color=black, color=white] (27:0.65) circle (1ex);
\fill [inner color=black, color=white] (63:0.65) circle (1ex);
\fill [inner color=black, color=white] (90:0.65) circle (1ex);

\fill [inner color=black!80, color=white] (0:1) circle (1ex);
\fill [inner color=black!80, color=white] (27:1) circle (1ex);
\fill [inner color=black!80, color=white] (63:1) circle (1ex);
\fill [inner color=black!80, color=white] (90:1) circle (1ex);

\fill [inner color=black!50, color=white] (0:1.3) circle (1ex);
\fill [inner color=black!50, color=white] (27:1.3) circle (1ex);
\fill [inner color=black!50, color=white] (63:1.3) circle (1ex);
\fill [inner color=black!50, color=white] (90:1.3) circle (1ex);

\fill [inner color=black!20, color=white] (0:1.6) circle (1ex);
\fill [inner color=black!20, color=white] (27:1.6) circle (1ex);
\fill [inner color=black!20, color=white] (63:1.6) circle (1ex);
\fill [inner color=black!20, color=white] (90:1.6) circle (1ex);

\draw (-0.35,0.65) node{$(x)$};
\draw[color=black!80] (-0.35,1) node{$(x^2)$};
\draw[color=black!50] (-0.35,1.3) node{$(x^3)$};
\draw[color=black!20] (-0.35,1.6) node{$(x^4)$};
\draw (0.65,-0.35) node{$(y)$};
\draw[color=black!80] (1,-0.35) node{$(y^2)$};
\draw[color=black!50] (1.3,-0.35) node{$(y^3)$};
\draw[color=black!20] (1.6,-0.35) node{$(y^4)$};
\draw[color=black] (27:0.65) ++ (0,0.25) node{$I$};
\draw[color=black!80] (27:1.0) ++ (0,0.25) node{$I^2$};
\draw[color=black!50] (27:1.3) ++ (0,0.25) node{$I^3$};
\draw[color=black!20] (27:1.6) ++ (0,0.25) node{$I^4$};

\fill [inner color=black, color=white] (45:2.3) circle (1.6ex);
\draw [fill=white,color=white] (45:2.3) circle (0.6ex);
\draw (45:2.3) node{$1$};
\draw (45:2.3) ++ (0.4,0) node{$(0)$};

\draw [color=black,rounded corners=3pt] (-0.15,-0.15) -- (-0.15,0.06) -- (0.15,0.06) -- (0.15,0.15) -- (1.5,0.15) -- (1.5,-0.15) -- cycle;
\end{tikzpicture}
\hspace{2cm}
\begin{tikzpicture}[scale=2]
\fill [color=black] (0, 0) circle (0.4ex);
\fill [color=black!80] (45:0.10) circle(0.2ex);
\fill [color=black!60] (45:0.17) circle(0.2ex);
\fill [color=black!40] (45:0.24) circle(0.2ex);
\fill [color=black!20] (45:0.31) circle(0.2ex);
\fill [color=black!10] (45:0.38) circle(0.2ex);

\fill [inner color=black, color=white] (0:0.65) circle (1ex);
\fill [inner color=black, color=white] (27:0.65) circle (1ex);
\fill [inner color=black, color=white] (63:0.65) circle (1ex);
\fill [inner color=black, color=white] (90:0.65) circle (1ex);

\fill [inner color=black!80, color=white] (0:1) circle (1ex);
\fill [inner color=black!80, color=white] (27:1) circle (1ex);
\fill [inner color=black!80, color=white] (63:1) circle (1ex);
\fill [inner color=black!80, color=white] (90:1) circle (1ex);

\fill [inner color=black!50, color=white] (0:1.3) circle (1ex);
\fill [inner color=black!50, color=white] (27:1.3) circle (1ex);
\fill [inner color=black!50, color=white] (63:1.3) circle (1ex);
\fill [inner color=black!50, color=white] (90:1.3) circle (1ex);

\fill [inner color=black!20, color=white] (0:1.6) circle (1ex);
\fill [inner color=black!20, color=white] (27:1.6) circle (1ex);
\fill [inner color=black!20, color=white] (63:1.6) circle (1ex);
\fill [inner color=black!20, color=white] (90:1.6) circle (1ex);

\draw (-0.35,0.65) node{$(x)$};
\draw[color=black!80] (-0.35,1) node{$(x^2)$};
\draw[color=black!50] (-0.35,1.3) node{$(x^3)$};
\draw[color=black!20] (-0.35,1.6) node{$(x^4)$};
\draw (0.65,-0.35) node{$(y)$};
\draw[color=black!80] (1,-0.35) node{$(y^2)$};
\draw[color=black!50] (1.3,-0.35) node{$(y^3)$};
\draw[color=black!20] (1.6,-0.35) node{$(y^4)$};
\draw[color=black] (27:0.65) ++ (0,0.25) node{$I$};
\draw[color=black!80] (27:1.0) ++ (0,0.25) node{$I^2$};
\draw[color=black!50] (27:1.3) ++ (0,0.25) node{$I^3$};
\draw[color=black!20] (27:1.6) ++ (0,0.25) node{$I^4$};

\fill [inner color=black, color=white] (45:2.3) circle (1.6ex);
\draw [fill=white,color=white] (45:2.3) circle (0.6ex);
\draw (45:2.3) node{$1$};
\draw (45:2.3) ++ (0.4,0) node{$(0)$};

\draw [color=black,rounded corners=8pt] (-0.15,-0.15) -- (-0.15,1.45) -- (0.15,1.45) -- (0.15,0.12) -- (0.95,0.65) -- (1.15,0.4) -- (0.85,0.25) -- (0.85,-0.15) -- cycle;

\end{tikzpicture}
\caption{The closures of the ideals $(y^3)$ and $(x^3 (x+y)^2 y)$ in $k\llbracket x, y\rrbracket$.  Here, $I$ now represents the specific closed prime $1$-ideal $(x+y)$. \label{fig:Closures}}
\end{figure}

We now seek to understand the intermediate ring $y^{-1} k\powser{x,y}$.  Just as before, multiplication by $y$ sends a closed set corresponding to $I$ to the closed set for $yI$.  To produce a portrait of this pipe ring, we identify the closed sets associated to $I$ and $J$ when
\[
\bigcup_k \overline{y^k I} = \bigcup_k \overline{y^k J}.
\]
In terms of the abbreviated portrait, this action admits a simple description: a closed set in the original portrait is sent to its union with $\{(y), (y^2), \ldots\}$.

%

However, there are an awful lot of ideals (i.e., closed sets) in $y^{-1}k\powser{x,y}$.  This uncomfortable fact is completely done away with by completing against the ideal $(x)$, producing the ring $(y^{-1}k\powser{x,y})^{\wedge}_{(x)}$.  The Weierstrass preparation theorem tells us that any element $f(x,y)$ of this ring takes the form
\[
f(x,y) = x^n \cdot g(x,y),
\] 
where $g(x,y)$ is a unit (or equivalently, considered as a power series in $x$ its constant coefficient is a unit).  It is instructive to note that this completion allows for power series that extend infinitely in both directions in $y$, provided that the coefficients of the negative powers of $y$ lie in increasing powers of the ideal $(x)$.  This is a proof-sketch of the fact that there is an isomorphism (of pipe rings)
\[
(y^{-1}k\powser{x,y})^{\wedge}_{(x)} \cong k\(y\)\powser{x}.
\]

Pictorially, this completion also has a very simple description: two closed sets are identified if they have the same intersection with the set $\{(x),(x^2),\ldots\}$.  For example, the closed sets defined by the ideals $(x^2(x+y)^3y^2)$ and $(x^2)$ become identified because $(x+y)^3y^2$ is now a unit.

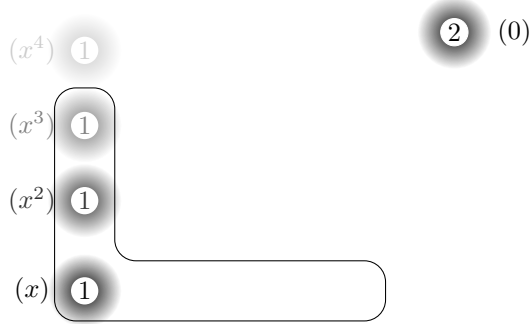
\begin{figure}
\begin{tikzpicture}[scale=2]

\fill [inner color=black, color=white] (90:0.0) circle (1.6ex);
\fill [inner color=black!80, color=white] (90:0.6) circle (1.6ex);
\fill [inner color=black!50, color=white] (90:1.1) circle (1.6ex);
\fill [inner color=black!20, color=white] (90:1.6) circle (1.6ex);

\fill [fill=white,color=white] (90:0) circle (0.6ex);
\fill [fill=white,color=white] (90:0.6) circle (0.6ex);
\fill [fill=white,color=white] (90:1.1) circle (0.6ex);
\fill [fill=white,color=white] (90:1.6) circle (0.6ex);

\draw (90:0) node{$1$};
\draw[color=black!80] (90:0.6) node{$1$};
\draw[color=black!50] (90:1.1) node{$1$};
\draw[color=black!20] (90:1.6) node{$1$};

\draw (-0.35,0.0) node{$(x)$};
\draw[color=black!80] (-0.35,0.6) node{$(x^2)$};
\draw[color=black!50] (-0.35,1.1) node{$(x^3)$};
\draw[color=black!20] (-0.35,1.6) node{$(x^4)$};

\draw [color=black,rounded corners=8pt] (-0.20,-0.20) -- (2.0,-0.20) -- (2.0,0.20) -- (0.20,0.20) -- (0.20,1.35) -- (-0.20,1.35) -- cycle;

\fill [inner color=black, color=white] (35:3) circle (1.6ex);
\draw [fill=white,color=white] (35:3) circle (0.6ex);
\draw (35:3) node{$2$};
\draw (35:3) ++ (0.4,0) node{$(0)$};
\end{tikzpicture}
\caption{Portrait of $(y^{-1} k\powser{x,y})^\wedge_{(x)}$, with the closure of $(x^3)$ indicated.  This closed set has been drawn with the distention to suggest that it's been inherited from the previous stage $y^{-1} k\powser{x,y}$ of the construction.}
\end{figure}

\subsection{The Lubin--Tate pipe rings}

For the remainder of the paper, we fix the following data: a finite field $k$ of positive characteristic $p$; a finite, positive integer $h$; a formal group $\Gamma$ of finite $p$-height $h$ over $k$; and a weakly decreasing infinite sequence of nonnegative integers \[h = h_0 \ge h_1 \ge h_2 \ge \cdots \ge h_n \ge \cdots \ge 0.\]  As indicated, we will very often refer to $h_0$ simply as $h$, as it plays a privileged role in the theory.

Having established the basics of a theory of rings which includes profinite rings, we turn to our second desideratum: the instantiation of $\pi_0 L_{K(t)} E_n$ and its associated localization morphism.  As in the classical case, our results concern formal groups and formal group laws, and we now introduce the appropriate notions.

\begin{definition}
The \textit{pipe spectrum} (or simply \textit{spectrum}) of an $n$-pipe ring $R$, denoted $\Spp(R)$, is defined to be the corepresentable functor $\Spp(R)(-) = \CatOf{PipeRings}_\infty(R, -)$.  Note that this coincides with $\Spec$ when restricted to $\CatOf{PipeRings}_{-1}$ and with $\Spf$ when restricted to $\CatOf{PipeRings}_0$.  We define the \textit{formal affine line} over $R$ to be \[\A^1_R = \Spp(R\powser{x}),\] where $R\powser{x}$ is the object of $\CatOf{PipeRings}_n = \Pro \Ind \CatOf{PipeRings}_{n-1}$ given by the sequential limit \[ \cdots \to R\{x^m, x^{m-1}, \ldots, x^0\} \to \cdots \to R\{x^2, x^1, x^0\} \to R\{x^1, x^0\} \to R\{x^0\},\] each of which carries the expected truncated polynomial structure.  (The notation is meant to indicate the distinction from the process of forming a polynomial ring and taking quotients, which again is not necessarily available to us.)  A \textit{$1$-dimensional formal variety} over $R$ is a functor $V: \CatOf{PipeRings}_\infty \to \CatOf{Sets}$ isomorphic to $\A^1_R$; a (smooth, $1$-dimensional, commutative) \textit{formal group} over $R$ is a group object structure on such a $V$; and a \textit{coordinatized formal group} over $R$ is an abelian group object structure on $\A^1_R$ itself.  Finally, the \textit{formal group law} associated to a coordinatized formal group is given by the representing power series in $\real R \powser{x,y}$ for the multiplication map \[\A^1_R \times \A^1_R \to \A^1_R.\]
\end{definition}

\begin{lemma}
A coordinatized formal group and a formal group law over a fixed pipe ring $R$ are equivalent data. \qed
\end{lemma}

\begin{definition}
Suppose that $R$ is a pipe ring, and select a formal group law $F$ over $R$ with $p$-series given by the formula \[[p]_F(x) = px + \sum_{i=2}^\infty a_i x^i.\]  Then, $F$ is said to be of \textit{$p$-height $h$} if $R$ is complete with respect to a closed ideal $I$, $I$ contains $p$, $I$ contains $a_i$ for every $i < p^h$, and the coefficient $a_{p^h}$ is invertible in $\real{R/I}$.
\end{definition}

Let us recall the basics of the Lubin--Tate theory of deformations of $p$-complete formal group laws of finite height.  The space of infinitesimal deformations of the formal group $\Gamma$ chosen above is represented by a formal affine variety $LT$ of dimension $h-1$ over the ring $\W_k$ of Witt vectors over $k$.  This space is not canonically coordinatized --- rather, several ``good'' choices of coordinates exist:
\begin{definition}
Let $\widetilde \Gamma$ be a versal deformation of $\Gamma$ to $LT$, and select a coordinate on $\widetilde \Gamma$.  Then, a \textit{Lubin--Tate coordinate system} on $LT$ is a set of coordinate functions $p, u_1, \ldots, u_{h-1}$ satisfying the following property for each $t$:
\[[p]_{\widetilde \Gamma}(x) \equiv u_t x^{p^t} \pmod{(p, u_1, \ldots, u_{t-1}, x^{1+p^t})}.\]
(In words, Lubin--Tate coordinates are defined inductively, with indeterminacy governed by a flag of ideals.)  Such coordinate systems always exist.  Moreover, the specific choice of versal deformation $\widetilde \Gamma$ does not affect whether a given coordinate system is a Lubin--Tate coordinate system.
\end{definition}

Consider the ring $\pi_0 E_h \cong \mathbb{W}_k \llbracket u_1, \ldots, u_{h-1} \rrbracket$, where the coordinates $u_1, \ldots, u_{h-1}$ form a Lubin--Tate coordinate system.  This ring is equipped with a topology specified by the ideal $I_h = (p, u_1, \ldots, u_{h - 1})$, which can be equivalently interpreted as a $0$-pipe structure.

\begin{definition}
Let $\pi_0 E_h$ denote the $0$-pipe ring above.  We inductively define the $n$-pipe ring \[\pi_0 L_{K(h_n)} \cdots L_{K(h_1)} E_h = (u_{h_n}^{-1}(\pi_0 L_{K(h_{n-1})} \cdots L_{K(h_1)} E_h))^\wedge_{I_{h_n}}\] by considering the rows of the following diagram as ind-systems and the diagram itself as a pro-system of its rows:
\begin{center}
\begin{tikzcd}
\pi_0 L_{K(h_{n-1})} \cdots L_{K(h_1)} E_{h} / I_{h_n} \arrow{r}{- \cdot u_{h_n}} \arrow[leftarrow]{d} & \pi_0 L_{K(h_{n-1})} \cdots L_{K(h_1)} E_{h} / I_{h_n} \arrow[leftarrow]{d} \arrow{r}{- \cdot u_{h_n}} & \cdots \\
\pi_0 L_{K(h_{n-1})} \cdots L_{K(h_1)} E_{h} / I_{h_n}^2 \arrow{r}{- \cdot u_{h_n}} & \pi_0 L_{K(h_{n-1})} \cdots L_{K(h_1)} E_{h} / I_{h_n}^2 \arrow{r}{- \cdot u_{h_n}} & \cdots. \\
\vdots \arrow{u} & \vdots \arrow{u} &
\end{tikzcd}
\end{center}
These quotients are defined using Proposition~\ref{prop:RigidQuotient}.  Using Hovey's result~\cite[Lemma 2.3]{Hovey}, the realization of this pipe ring is indeed the zeroth homotopy ring of the indicated ring spectrum.
\end{definition}

\begin{lemma}
The $n$-pipe ring $\pi_0 L_{K(h_n)} \cdots L_{K(h_1)} E_{h}$ is bifine. \qed
\end{lemma}

\begin{proposition}
The universal formal group law over $\pi_0 E_{h}$ pulled back to $\pi_0 L_{K(h_n)} \cdots L_{K(h_1)} E_{h}$ is of $p$-height $h_n$.
\end{proposition}
\begin{proof}
This follows from the definition of a Lubin--Tate coordinate system, along with the fact that completion and localization have the expected actions on realizations.
\end{proof}

As one would hope, the $n$-pipe ring $\pi_0 L_{K(h_n)} \ldots L_{K(h_1)} E_{h}$ does not depend upon the choice of coordinates.  Of course, this must be true if we expect to give a formal-geometric interpretation of its functor of points and if the $n$-pipe ring in question indeed comes from a homotopy-theoretic construction which \textit{a priori} has nothing to do with coordinates.

\begin{theorem}
The $n$-pipe ring $\pi_0 L_{K(h_n)} \cdots L_{K(h_1)} E_{h}$ is independent of the choice of Lubin--Tate coordinates $u_1, \ldots, u_{h-1}$.
\end{theorem}
\begin{proof}
Let $X_{n-1}^u$ denote the pipe ring $\pi_0 L_{K(h_{n-1})} \cdots L_{K(h_1)} E_{h}$ as presented with $u$-coordinates.  Suppose that $X_{n-1}^v$ is a second such ring, presented in $v$-coordinates, and suppose that the change of coordinates isomorphism $\phi_0: X_0^u \to X_0^v$ has been shown to extend uniquely to an isomorphism $\phi_{n-1}: X_{n-1}^u \to X_{n-1}^v$.  The remainder of the situation is then summarized in the following diagram:
\begin{center}
\begin{tikzcd}
X_{n-1}^u \arrow{r}{\phi_{n-1}} \arrow{d} \arrow{rdd} & X_{n-1}^v \arrow{d} \\
u_{h_n}^{-1} X_{n-1}^u \arrow{d} \arrow[dashed]{rd} & v_{h_n}^{-1} X_{n-1}^v \arrow{d} \\
X_n^u = (u_{h_n}^{-1} X_{n-1}^u)^\wedge_{I_{h_n}^u} \arrow[dashed]{r}{\phi_n} & (v_{h_n}^{-1} X_{n-1}^v)^\wedge_{I_{h_n}^v} = X_n^v.
\end{tikzcd}
\end{center}
We seek to construct the dashed arrows.  (As we will see, in general there is not a map $u_{h_n}^{-1} X_{n-1}^u \to v_{h_n}^{-1} X_{n-1}^v$.)  All of the vertical maps are injective on set-theoretic points, and so the question of the existence of the diagonal dashed arrow is determined by whether $u_{h_n}$ is sent to an invertible element in the target.  Since
\[\phi_{n-1}(u_{h_n}) = v_{h_n} + c\]
for some $c \in I^v_{h_n}$, then $\phi_{n-1}(h_n)$ has an inverse in $X^v_n$ given by
\[\phi_{n-1}(u_{h_n})^{-1} = \frac{1}{v_{h_n} + c} = \frac{v_{h_n}^{-1}}{1 + v_{h_n}^{-1}c} =  \sum_{j=0}^\infty (-1)^j c^j v_{h_n}^{-1-j}.\]
(Note that the existence of this infinite sum depends on the completeness of $X^v_n$ against the ideal $I^v_{h_n}$.)  This gives us the diagonal dashed arrow.  Finally, we obtain $\varphi_n$ from the fact that $\varphi_{n-1}(I^u_{h_n}) \subset I^v_{h_n}$.

Of course, switching the two pipe rings and applying this same method produces an inverse to $\phi_n$.
\end{proof}

\begin{remark}
These pipe rings $\pi_0 L_{K(h_n)} \cdots L_{K(h_1)} E_h$ come equipped with a the structure of a rigid pipe ring, by definition.  This proposition also serves to show that the rigid structure is independent of the choice of Lubin--Tate coordinate system.
\end{remark}

\section{A Lubin--Tate theorem}

We remind the reader that we have fixed the following data: a finite field $k$ of positive characteristic $p$; a finite, positive integer $h$; a formal group $\Gamma$ of finite $p$-height $h$ over $k$; and a weakly decreasing infinite sequence of nonnegative integers \[h = h_0 \ge h_1 \ge h_2 \ge \cdots \ge h_n \ge \cdots \ge 0.\]  Again, we will often refer to $h_0$ simply as $h$, as it plays a privileged role in the theory.

We now explore the moduli problem represented by the pipe spectrum \[\Spp \pi_0 L_{K(h_n)} \cdots L_{K(h_1)} E_{h}\] described in the previous section.  To do so, we need one further definition:
\begin{definition}
A \textit{staged deformation of $k$} is a staged ring $R_*$ where $R_0$ is a complete local ring with residue field $k$.  The category of staged deformations is denoted $\CatOf{StagedDef}$.\footnote{``Staged deformation'' is likely something of a misnomer, as we seem to be missing a condition on $R_n$ for $n > 0$ to link up with the word ``deformation''.  This won't affect our main theorem, though.} 
\end{definition}

\begin{definition}[Staged Lubin--Tate moduli problem]
Select a formal group law $H$ for the fixed formal group $\Gamma$.  Let $\M = \M_{H/k}$ be the moduli problem on $\CatOf{StagedDef}$ which assigns to a staged deformation \[R_0 \xrightarrow{i_1} R_1 \xrightarrow{i_2} \cdots \xrightarrow{i_n} R_n \xrightarrow{i_{n+1}} \cdots \] the following groupoid:
\begin{itemize}
\item The objects are given by commuting diagrams of the form
\begin{center}
\begin{tikzcd}
H \arrow{r} \arrow{d} & F_0 \arrow{d} & F_1 \arrow{l} \arrow{d} & \cdots \arrow{l} & F_n \arrow{l} \arrow{d} & \cdots \arrow{l} \\
\Spp k \arrow{r} & \Spp R_0 & \Spp R_1 \arrow{l} & \cdots \arrow{l} & \Spp R_n \arrow{l} & \cdots \arrow{l},
\end{tikzcd}
\end{center}
where the $F_n$ are coordinatized formal groups over $\Spp R_n$ of height $h_n$, each square is a pullback square, and the arrows in the top row are morphisms of coordinatized formal groups (i.e., the formal group law $F_n$ is specified by pushing forward $F_{n-1}$ along $i_n$).
\item Morphisms in the groupoid correspond to commuting diagrams of shape:
\begin{center}
\begin{tikzcd}[column sep=1.70em]
& H \arrow{ld} & & F_0 \arrow[leftarrow]{ll} \arrow[leftarrow]{rr} \arrow{ld} & & F_1 \arrow[leftarrow]{rr} \arrow{ld} & & \cdots \\
\Spp k & & \Spp R_0 \arrow[leftarrow]{ll} \arrow[leftarrow]{rr} & & \Spp R_1 \arrow[leftarrow]{rr} & & \cdots \\
& H \arrow{lu} \arrow[double, -, crossing over]{uu} & & F'_0 \arrow[leftarrow]{ll} \arrow[leftarrow]{rr} \arrow{lu} \arrow[leftarrow, crossing over]{uu} & & F'_1 \arrow[leftarrow]{rr} \arrow{lu} \arrow[leftarrow, crossing over]{uu} & & \cdots
\end{tikzcd}
\end{center}
where first vertical arrow $H \to H$ is an equality and the other vertical arrows are formal group isomorphisms (which need not respect the chosen coordinatizations).  Note that the pullback condition for the object diagrams implies that the isomorphism $F_0 \to F'_0$ determines the isomorphisms $F_n \to F'_n$ for all $n \ge 1$.
\end{itemize}
\end{definition}

\begin{theorem}\label{thm:LubinTatePair}
Restricted to fine staged deformations (i.e., staged deformations of fine pipe rings), the moduli problem $\M$ is essentially discrete and is corepresented by the bifine staged deformation \[\pi_0 E_{h} \to \pi_0 L_{K(h_1)} E_{h} \to \cdots \to \pi_0 L_{K(h_n)} \cdots L_{K(h_1)} E_{h} \to \cdots.\]
\end{theorem}
\begin{proof}
From here on, we refer to the staged ring in the theorem statement as $X_0 \to X_1 \to X_2 \to \cdots$ and the formal group over $X_n$ simply as \[\G_n = \G_{L_{K(h_n)} \cdots L_{K(h_1)} E_{h}}.\]  The classical Lubin--Tate moduli problem is essentially discrete~\cite[Theorem 4.4]{Rezk}, and since we have restricted to fine rings and the first vertical arrow in the diagram above determines the latter ones by pullback, it is true for $\M$ as well.\footnote{More concisely, $\M(R_0 \to R_1 \to \cdots)$ is a subgroupoid of the classical Lubin--Tate moduli functor evaluated on $R_0$ alone.}  This means that $\M$ is valued in discrete groupoids, and so it only remains to construct the left-hand natural isomorphism in \[\CatOf{StagedDef}\left(\begin{array}{c} X_0 \\ \downarrow \\ X_1 \\ \downarrow \\ \vdots \end{array}, -\right) \xrightarrow{\cong} \pi_0 \M \xleftarrow{\simeq} \M.\]

To accomplish this, we must produce for a point $(F_0, F_1, \ldots) \in \M(R_0 \xrightarrow{i_1} R_1 \xrightarrow{i_2} \cdots)$ the data of morphisms of $n$-pipe rings $f_n: X_n \to R_n$, along with commuting isomorphisms of formal groups $\phi_n: F_n \to f_n^* \G_n$ over $\Spp R_n$.  Altogether, this data fits into the following diagram:
\begin{center}
\begin{tikzcd}[column sep=0.9em]
& H \arrow{ldd} & & F_0 \arrow[leftarrow]{ll} \arrow[leftarrow]{rr} \arrow{ldd} & & \cdots \arrow[leftarrow]{rr} & & F_{n-1} \arrow[leftarrow]{rr} \arrow{ldd} & & F_n \arrow[leftarrow]{rr} \arrow{ldd} & & \cdots \\
\\
\Spp k & & \Spp R_0 \arrow[leftarrow]{ll} \arrow[leftarrow]{rr} \arrow{dd} & & \cdots \arrow[leftarrow]{rr} & & \Spp R_{n-1} \arrow[leftarrow]{rr} \arrow{dd} & & \Spp R_n \arrow[leftarrow]{rr} \arrow{dd} & & \cdots \\
& H \arrow{lu} \arrow[double, -, crossing over]{uuu} & & f_0^* \G_0 \arrow[crossing over,leftarrow]{ll} \arrow[crossing over,leftarrow]{rr} \arrow{lu} \arrow[leftarrow, crossing over]{uuu}[description]{\phi_0} & & \cdots \arrow[crossing over,leftarrow]{rr} & & f_{n-1}^* \G_{n-1} \arrow[leftarrow,crossing over]{rr} \arrow{lu} \arrow[leftarrow, crossing over]{uuu}[description]{\phi_{n-1}} & & f_n^* \G_n \arrow[leftarrow]{rr} \arrow{lu} \arrow[leftarrow, crossing over]{uuu}[description]{\phi_n} & & \cdots \\
\Spp k \arrow[double,-]{uu} & & \Spp X_0 \arrow[leftarrow]{ll} \arrow[leftarrow]{rr} & & \cdots \arrow[leftarrow]{rr} & & \Spp X_{n-1} \arrow[leftarrow]{rr} & & \Spp X_n \arrow[leftarrow]{rr} & & \cdots \\
& H \arrow{ul} \arrow[crossing over, double, -]{uu} & & \G_0 \arrow[leftarrow]{ll} \arrow[leftarrow]{rr} \arrow{ul} \arrow[crossing over, leftarrow]{uu} & & \cdots \arrow[leftarrow]{rr} & & \G_{n-1} \arrow[leftarrow]{rr} \arrow{ul} \arrow[crossing over, leftarrow]{uu} & & \G_n \arrow[leftarrow]{rr} \arrow{ul} \arrow[crossing over, leftarrow]{uu} & & \cdots.
\end{tikzcd}
\end{center}

The morphisms $f_0$ and $\phi_0$ are handled automatically by classical Lubin--Tate theory~\cite[Proposition 1.1]{LubinTate}, so our task is to inductively construct the rest.  To extend the map $f_{n-1}$ to a map $f_n$, we need the following result, proven below as Lemma~\ref{thm:FGLIsoPreservesCompleteness}: if the first $(h_n-1)$ coefficients of the $p$-series of a formal group law are in an ideal against which the pipe ring is complete and the $h_n$th coefficient is invertible, then the same is true for any isomorphic formal group law.

Postponing that proof for now, let us apply this result to our situation.  The isomorphism $\phi_{n-1}$ pulls back along $i_n$ to give an isomorphism \[i_n^* F_{n-1} \to i_n^* f_{n-1}^* \G_{n-1}.\]  This isomorphism will become $\phi_n$ once we identify the source and target with the expected values, and so we refer to it as $\phi_n$ below.  By construction, the source can be identified with $F_n$, and because we are working in a Lubin--Tate coordinate system its associated $p$-series can be written as \[[p]_{F_n}(x) = px + \sum_{i=2}^\infty a_i x^i,\] with $R_n$ complete against the coefficients $p$ and $a_i$ for $i < p^{h_n}$ and with $a_{p^{h_n}}$ a unit.  Define $\a_n$ to be the ideal generated by $a_i$ for $i < p^{h_n}$.  Applying Lemma~\ref{thm:FGLIsoPreservesCompleteness} to the isomorphism $\phi_n$, we have that the coefficients of the $p$-series of $i_n^* f_{n-1}^* \G_{n-1}$ lie in $\a_n$ up through order $p^{h_n} - 1$ and that the coefficient of $x^{p^{h_n}}$ is invertible.  Hence, the composite $i_n \circ f_{n-1}$ extends uniquely along both the localization $X_{n-1} \to u_{h_n}^{-1} X_{n-1}$ and the completion $u_{h_n}^{-1} X_{n-1} \to (u_{h_n}^{-1} X_{n-1})^\wedge_{I_{h_n}} = X_n$, using the definitions of the localization and completion as ind- and pro-objects.  This gives the desired map $f_n$.
\end{proof}

\begin{lemma}\label{thm:FGLIsoPreservesCompleteness}
Let $G$ and $G'$ be formal group laws over a ring $S$, and let $\phi: G' \to G$ be an isomorphism.  Write the $p$-series of $G$ and $G'$ as
\begin{align*}
[p]_G(x) & = px + \sum_{i=2}^\infty a_i x^i, \\
[p]_{G'}(x) & = px + \sum_{i=2}^\infty a'_i x^i.
\end{align*}
Suppose that $S$ is complete with respect to some ideal $\a$.  If $G$ has the property that $a_i \in \a$ for $i < n$ and $a_n \in R^\times$, called (*), then $G'$ has property (*) for the $a'_i$.\footnote{Lemma~\ref{thm:FGLIsoPreservesCompleteness} is of course not actually a statement about formal group laws and their $p$-series.  It is equally true when $\phi$ is an invertible power series intertwining two series $F$ and $F'$ by the formula $F'(x) = \phi^{\circ(-1)}(F(\phi(x)))$.  We state it in the specific language of formal group laws we will need, for ease of reference.}
\end{lemma}
\begin{proof}
Write $\phi(x) = \sum_{j=1}^\infty b_j x^j$ with $b_1 \in R^\times$.   We would like to prove that property (*) holds for the series \[[p]_{G'}(x) = \phi^{\circ(-1)}([p]_G(\phi(x))).\]  We begin by proving that property (*) holds for the argument $[p]_G(\phi(x))$, which we expand as
\begin{align*}
[p]_G(\phi(x)) & = \sum_{i=1}^\infty a_i \cdot \phi(x)^i \\
& = \sum_{i=1}^\infty a_i \cdot \left(\sum_{j=1}^\infty b_j x^j\right)^i \\
& = \sum_{i=1}^\infty \sum_{j_1, \ldots, j_i > 0} a_i b_{j_1} \cdots b_{j_i} x^{j_1+\cdots + j_i} \\
& = \sum_{m=1}^\infty \beta_m x^m.
\end{align*}
First, it is clear that $\beta_m \in \a$ for $m<n$, since the coefficients in the penultimate line that sum to give $\beta_m$ all contain some factor $a_i$ with $i<n$, which by assumption lives in $\a$.  Furthermore, among the coefficients summing to give $\beta_n$, we have $a_n b_1^n \in R^\times$ when $i=n$ and $b_{j_1}=\ldots = b_{j_i}=1$, and then all the rest contain some factor $a_i$ with $i<n$ and hence live in $\a$.  Thus $\beta_n \in R^\times$, as the sum of a unit and an element of $\a$ is again a unit since $S$ is complete with respect to $\a$.  So the argument $[p]_G(\phi(x))$ does indeed satisfy property (*).

Let us write $\phi^{\circ (-1)} = \sum_{j=1}^\infty c_j x^j$, which has $c_1 = b_1^{-1} \in R^\times$.  We can now see that
\[ [p]_{G'}(x) = \phi^{\circ(-1)}\left( [p]_G(\phi(x)) \right) = \sum_{j=1}^\infty c_j \left( \sum_{m=1}^\infty \beta_mx^m \right)^j = \sum_{j=1}^\infty \sum_{m_1,\ldots,m_j>0} c_j \beta_{m_1} \cdots \beta_{m_j} x^{m_1+\cdots + m_j}
\]
satisfies property (*) by an identical argument to the one above.
\end{proof}

\begin{corollary}\label{cor:HeightIsIsoStable}
The $p$-height is independent of the chosen coordinate on a given formal group. \qed
\end{corollary}

\section{Final remarks}


\begin{remark}
The language of $p$-divisible groups may also give a natural setting for the hypotheses of Theorem~\ref{thm:LubinTatePair}.  The relevant feature of a $p$-divisible group is that it comes with two invariants: its height, which is stable under pullback along \emph{any} morphism of schemes, and its formal height, which behaves as height for formal groups --- it is stable only for pullbacks along continuous maps in the relevant adic topology and can decrease for a general map.  The manifestation of $p$-divisible groups in algebraic topology has been partially explored in the third author's thesis~\cite{StapletonCharacters}.  The following algebro-topological formulas are shown there to produce relevant $p$-divisible groups with differing height and formal height:
\begin{align*}
\G_{E_{h}}[p^\infty] & = \colim_j \Spf_{I_{h}} \pi_0 \left[ E_{h}^{B\Z/p^j} \right], \\
i_1^*(\G_{E_{h}}[p^\infty]) & = \colim_j \Spf_{I_{h_1}} \pi_0 \left[ L_{K(h_1)} \left(E_{h}^{B\Z/p^j}\right) \right], \\
(i_1^* \G_{E_{h}})[p^\infty] & = \left(i_1^*(\G_{E_{h}}[p^\infty])\right)^\circ = \colim_j \Spf_{I_{h_1}} \pi_0 \left[ (L_{K(h_1)} E_{h})^{B\Z/p^j} \right].
\end{align*}
\end{remark}

\begin{remark}
The Landweber exact functor theorem lifts our purely algebraic theorem to a statement about certain (homotopy) $BP$-algebra spectra.  An even $BP_*$-algebra $R$ is said to be \textit{Landweber flat} when the sequence $(p, v_1, \ldots)$ is regular on $R$, i.e., when the maps \[R / (p, \ldots, v_{n-1}) \xrightarrow{\cdot v_n} R/(p, \ldots, v_{n-1})\] are injective for all $n \ge 0$.  This is true not just of $\pi_0 E_{h}$ but also of $\pi_0 L_{K(h_n)} \cdots L_{K(h_1)} E_{h}$, and so Landweber's theorem along with a result of Hovey and Strickland~\cite[Proposition 2.20]{HoveyStrickland} yields a staged $BP$-algebra $E_h \to L_{K(h_1)} E_{h} \to \cdots$.  Hence, when the formal group laws associated to a point in $\M(R_*)$ imbue each realized pipe ring $\real{R_n}$ with the structure of a Landweber flat $BP_*$-algebra, we get a unique morphism of sequences of the corresponding $BP$-algebras.
\end{remark}

\begin{remark}[\textit{Ceci n'est pas une pipe}]
Our original formulation of this moduli problem used a different source category.  Let $\CatOf{StagedDef}'$ be defined as the category of squares of functors
\begin{center}
\begin{tikzcd}
\N^{\mathrm{discrete}} \arrow{r} \arrow{d} & \N \arrow{d} \\
\CatOf{CompleteTopologicalRings} \arrow{r} & \CatOf{Rings},
\end{tikzcd}
\end{center}
with morphisms commuting natural transformations.  Writing $R'$ for the functor $R': \N \to \CatOf{Rings}$, this is to say: the morphism of rings $R'_n \to R'_{n+1}$ \emph{within} a sequence need not be continuous, but we require that a morphism $R'_n \to S'_n$ across sequences be continuous.  Using this category, we also define a second moduli problem $\M'$, identical in definition to $\M$ except that it uses $\CatOf{StagedDef}'$ as a source.  Now suppose that we have a point $(F'_0, F'_1, \ldots) \in \M'(R'_0 \to R'_1 \to \cdots)$.  Then we can also given a Lubin--Tate type theorem for $\M'$: there is again a unique sequence of compatible maps \[f'_n: \pi_0 L_{K(h_n)} \cdots L_{K(h_1)} E_h \to R'_n,\] continuous in the $I'_{h_n}$-adic topology, expressing $F'_n$ as the pullback $(f'_n)^* \G_n$.

However, this problem ``factors'' through the moduli $\M$ considered here.  Recall that there is a set-theoretic realization morphism $\CatOf{StagedDef} \to \CatOf{StagedDef}'$.  One can construct a staged pipe ring deformation $R_* \in \CatOf{StagedDef}$ which begins with $R_0 = R'_0$ and which comes equipped with a map ${\real R}_* \to R'_*$ in $\CatOf{StagedDef}'$ factoring the Lubin--Tate map $f'_*$.  The system $R_*$ is formed by iterated pushouts:
\begin{center}
\begin{tikzcd}[column sep=.2em]
& R'_0 \arrow{rr} & & R'_1 \arrow{rr} & & R'_2 \arrow{rr} & & \cdots \\
R_0 \arrow{rr} \arrow{ru} \arrow[leftarrow]{rd} & & \pi_0 L_{K(h_1)} E_h \underset{\pi_0 E_h}{\hat\otimes} R_0 \arrow{rr} \arrow{ru} \arrow[leftarrow]{rd} & & \pi_0 L_{K(h_2)} L_{K(h_1)} E_h \underset{\pi_0 E_h}{\hat\otimes} R_0 \arrow{ru} \arrow[leftarrow]{rd} \arrow{rr} & & \cdots \\
& \pi_0 E_h \arrow{rr} \arrow[crossing over]{uu} & & \pi_0 L_{K(h_1)} E_h \arrow{rr} \arrow[crossing over]{uu} & & \pi_0 L_{K(h_2)} L_{K(h_1)} E_h \arrow{rr} \arrow[crossing over]{uu} & & \cdots
\end{tikzcd}
\end{center}

It is also worth pointing out that the moduli problem $\M'$ is often empty: since there is no topological restriction on the maps $R'_n \to R'_{n+1}$, there is no reason to think that an ideal of definition of $R'_{n+1}$ will preimage to something contained in the ideal of definition of $R'_n$.  In this situation, no formal group law can have the height property demanded by $\M'$.
\end{remark}

\begin{remark}
One obvious place where the iterated localization $L_{K(h_n)} \cdots L_{K(h_1)} E_h$ appears is in the chromatic fracture cube for $E_h$.  Unfortunately, since $E_h$ is $K(h)$-local, our study of the iterated localizations has little new to add to the discussion.  For example, when $h = 2$, the chromatic fracture cube is the limit diagram
\begin{center}
\begin{tikzcd}[row sep=1em]
& L_2 E_2 \arrow{rr} \arrow{ld} \arrow{dd} & & L_{K(1)} E_2 \arrow{ld} \arrow{dd} \\
L_{K(2)} E_2 \arrow[crossing over]{rr} \arrow{dd} & & L_{K(1)} L_{K(2)} E_2 \\
& L_{K(0)} E_2 \arrow{rr} \arrow{ld} & & L_{K(0)} L_{K(1)} E_2 \arrow{ld} \\
L_{K(0)} L_{K(2)} E_2 \arrow{rr} & & L_{K(0)} L_{K(1)} L_{K(2)} E_2. \arrow[crossing over,leftarrow]{uu}
\end{tikzcd}
\end{center}
Since $E_2 = L_2 E_2 = L_{K(2)} E_2$, each back-to-front morphism is already an equivalence, meaning that the pullback condition on the fracture cube is vacuously satisfied.  Using a fracture cube to compute $L_t E_n$ for $t < n$ is also a dead end; consider for example the limit square
\begin{center}
\begin{tikzcd}
L_1 E_2 \arrow{r} \arrow{d} & L_{K(1)} E_2 \arrow{d} \\
L_{K(0)} E_2 \arrow{r} & L_{K(0)} L_{K(1)} E_2.
\end{tikzcd}
\end{center}
Since we know that all of these besides $L_1E_2$ are even-concentrated, we obtain an exact sequence
\[
0 \to \pi_* L_1 E_2 \to p^{-1}(\Z_p\powser{u_1}) \oplus (\Z_p\(u_1\))^\wedge_p \to p^{-1}\left(\Z_p\(u_1\))^\wedge_p\right) \to \pi_{*-1} L_1 E_2 \to 0.
\]
From this we can see that $L_1 E_2$ has loads of odd-dimensional homotopy, which prevents the direct application of formal geometry as we understand it.
\end{remark}

\bibliographystyle{plain}
\bibliography{lubintate}

\end{document}